\documentclass[11pt,a4paper,draft]{article}

\usepackage{amsmath,amsthm,amsopn,amsfonts,amssymb,dsfont,color}
\usepackage[dvips]{graphicx}
\usepackage{psfrag}
\usepackage{enumerate}

\usepackage{a4,a4wide}

\def\ep{{\varepsilon}}
\def\R{\mathbb R}

\def\J{\widehat J}

\def\E{e^{t(\J(\xi)-1)}}
\def\uzero{\widehat{u_0}}
\def\dirac{\delta _0}

\newtheorem{theo}{\textbf{Theorem}}[section]

\newtheorem{lem}[theo]{\textbf{Lemma}}

\newtheorem{cor}[theo]{\textbf{Corollary}}

\newtheorem{assumption}[theo]{\textbf{Assumption}}
\newtheorem{rem}[theo]{\textbf{Remark}}

\newcommand{\monstre}{\left(1-\int _{\vert y\vert\geq m\tau ^{1/\beta}}\psi (T,x-y)dy\right)^{-p}-1}

\title{ Fujita blow up phenomena and hair trigger effect: the role of dispersal tails}
\date{}

\begin{document}

\maketitle

\begin{center}
{\large\bf Matthieu Alfaro \footnote{ IMAG, Universit\'e de
Montpellier, CC051, Place Eug\`ene Bataillon, 34095 Montpellier
Cedex 5, France. E-mail: matthieu.alfaro@umontpellier.fr}}\\
[2ex]

\end{center}


\tableofcontents

\vspace{10pt}

\begin{abstract} We consider the nonlocal diffusion equation $\partial _t u=J*u-u+u^{1+p}$ in the whole of $\R ^N$. We prove that the Fujita exponent  dramatically depends on the behavior of the Fourier transform of the kernel $J$ near the origin, which is linked to the tails of $J$. In particular, for compactly supported  or exponentially bounded  kernels, the Fujita exponent is the same as that of the nonlinear Heat equation $\partial _tu=\Delta u+u^{1+p}$. On the other hand, for kernels with algebraic tails, the Fujita exponent  is either of  the Heat type or of some related Fractional type, depending on the finiteness of the second moment of $J$. As an application of the result in population dynamics models, we discuss the hair trigger effect for $\partial _t u=J*u-u+u^{1+p}(1-u)$.
 \\


\end{abstract}

\section{Introduction}\label{s:intro}

In this work we consider solutions $u(t,x)$ to the nonlinear ($p>0$) integro partial differential equation
\begin{equation}
\label{eq}
\partial _t u=J*u-u+u^{1+p} \quad \text{ in } (0,\infty)\times \R^{N},
\end{equation}
in any dimension $N\geq 1$. Equation \eqref{eq} is supplemented with a nonnegative and nontrivial initial data, and we aim at determining the so-called {\it Fujita exponent} $p_F$, that is the value of $p$ that separates \lq\lq systematic blow up solutions'' from \lq\lq blow up solutions vs global and extincting solutions'' (see below for details). We shall prove that the Fujita exponent dramatically depends on the behavior of the Fourier transform of the kernel $J$ near the origin, which is linked to the tails of $J$. Depending on these tails, it turns out that the Fujita phenomenon in \eqref{eq} can be similar to  that of the nonlinear Heat equation, or to that of a related nonlinear Fractional equation.

As an application of our main result, we consider
\begin{equation}
\label{eq2}
\partial _t u=J*u-u+u^{1+p}(1-u) \quad \text{ in } (0,\infty)\times \R^{N},
\end{equation}
which serves as a population dynamics model where both long range dispersal (via the kernel $J$) and a weak Allee effect (via the degeneracy of the steady state $u\equiv 0$, due to $p>0$)  are taken into account. 
Depending on the balance between the tails of $J$ and the strength of the Allee effect, we discuss the so-called {\it hair trigger effect} --- meaning that any small perturbation from $u\equiv 0$ drive the solution to $u\equiv 1$--- or the possibility of extincting solutions.

\medskip

In his seminal work \cite{Fuj-66}, Fujita considered solutions $u(t,x)$ to the nonlinear Heat equation 
\begin{equation}
\label{eq-fujita}
\partial _t u=\Delta u+u^{1+p}\quad \text{ in } (0,\infty)\times \R^{N},
\end{equation}
supplemented with a nonnegative and nontrivial initial data. For such a problem, the Fujita exponent is $p_F=\frac 2 N$. Precisely,  if $0<p\leq p_F$ then any solution blows up in finite time; if $p>p_F$ then solutions with large initial data blow up in finite time whereas solutions with small initial data are global in time and go extinct as $t\to\infty$. For a precise statement we refer to \cite{Fuj-66} for the cases $0<p< p_F$ and $p>p_F$. The critical case $p=p_F$ is studied in \cite{Hay-73} when $N=1,2$, in \cite{Kob-Sir-Tun-77} when $N\geq 3$, and in \cite{Wei-81} via a direct and simpler approach. Let us observe that, as well-known, solutions to the Heat equation  $\partial _t u=\Delta u$ tend to zero as $t\to\infty$ like  $\mathcal O\left(t^{-\frac  N 2}\right)$, which is a formal  argument to guess $p_F=\frac 2 N$.

Since then, the Fujita phenomenon has attracted much interest and the literature on refinements of the results or on various local variants of equation \eqref{eq-fujita} is rather large. Let us mention for instance the works \cite{Wei-81}, \cite{Lev-90}, \cite{Den-Lev-00}, \cite{Mit-Pok-01}, or \cite{Qui-Sou-book} for an overview, and the references therein.

When the Laplacian diffusion operator is replaced  by the Fractional Laplacian, the situation is also well understood: the Fujita exponent for
\begin{equation}
\label{eq-sugitani}
\partial _t u=-(-\Delta)^{s/2}u+u^{1+p}\quad \text{ in } (0,\infty)\times \R^{N}, \quad 0<s\leq 2,
\end{equation}
is $p_F=\frac s N$. We refer to the work of Sugitani \cite{Sug-75}. See also, among others,  \cite{Bik-Lop-Wak-02}  for a probabilistic approach, and \cite{Gue-Kir-99} for a  variant of \eqref{eq-sugitani}. Let us observe that, as well-known,  solutions to   the Fractional diffusion equation $\partial _t u=-(-\Delta)^{s/2} u$ tend to zero as $t\to\infty$ like  $\mathcal O\left(t^{-\frac N s}\right)$, which is again a formal  argument to guess $p_F=\frac s N$.

As far as we know, much less is known for the nonlocal equation \eqref{eq}. Let us mention the work of Garc{\'{\i}}a-Meli{\'a}n  and Quir{\'o}s \cite{Gar-Qui-10} (and \cite{Yan-Zho-Zhe-14} for a variant) who treat the case of compactly suported dispersal kernel $J$. In such a situation, the Fujita exponent for \eqref{eq} is the same as that of \eqref{eq-fujita}, namely $p_F=\frac 2 N$. In order to take into account rare long-distance dispersal events, which are relevant in many population dynamics models (seeds dispersal for instance), we allow in this work kernels $J$ which have nontrivial tails. Two typical situations are when $J$ has (light) exponential tails or (heavy) algebraic tails, the latter case meaning
\begin{equation}
\label{tails}
J(x)\sim \frac C{\vert x\vert ^{\alpha}} \quad \text{ as } \vert x\vert \to \infty, \quad\text{ with } \alpha >N.
\end{equation}
Owing to the decay of solutions to $\partial _tu=J*u-u$ proved by Chasseigne, Chaves and Rossi  \cite{Cha-Cha-Ros-06}, we 
 guess that $p_F=\frac 2 N$ in the exponential case, whereas
\begin{equation}
\label{exponent}
p_F=\begin{cases}
\frac \alpha N -1  &\text{ if } N<\alpha\leq N+2 \\
\frac 2 N  &\text{ if } \alpha>N+2,\end{cases}
\end{equation}
in the algebraic case \eqref{tails}. In other words, in the algebraic case $\alpha>N+2$ the Fujita exponent is of the Heat type \eqref{eq-fujita} (and so in the exponential case), but in the algebraic case $N<\alpha \leq N+2$ the Fujita exponent becomes of the Fractional type \eqref{eq-sugitani} with $s=\alpha -N \in (0,2]$.  This is the role of the present paper to prove, among others, these results.

Let us comment on some technical difficulties arising from \eqref{eq}. Notice first that, as far as \eqref{eq-fujita} and \eqref{eq-sugitani} are concerned, some self similarity properties of both the Heat kernel and the  fundamental solution associated to the Fractional Laplacian may be quite helpful, as seen in \cite{Sug-75} or \cite{Wei-81}. Those self similarity properties are not shared by  the fundamental solution of $\partial _t u=J*u-u$. Secondly, notice that, when $J$ is compactly supported, the underlying nonlocal eigenvalue problem to \eqref{eq} is rather well understood \cite{Gar-Ros-09} and the authors in \cite{Gar-Qui-10} took advantage of  its rescaling properties. As far as we know, such informations are not available for more general dispersal kernels, as those we consider. We therefore have to adapt some technics, in particular when dealing with blow up phenomena. 

\medskip
We now discuss the hair trigger effect in some population dynamics models. Let us start with the Fisher-KPP equation 
$$
\partial _t u=\Delta u+u(1-u),
$$
which was introduced \cite{Fis-37}, \cite{Kol-Pet-Pis-37}, to model the spreading of advantageous genetic
features in a population. From the linear instability of the steady state $u\equiv 0$, it is well known that any solution $u(t,x)$ to the Fisher-KPP equation,
with a nonnegative and nontrivial initial data, tend to $1$ as $t\to \infty$, locally uniformly in $x\in\R ^ N$. This is referred to as the hair trigger effect.

In order to take into account a weak Allee effect, meaning that the growth per capita is no longer maximal at small densities, one may consider
\begin{equation}
\label{eq-aronson}
\partial _t u=\Delta u+u^{1+p}(1-u),
\end{equation}
where $p>0$. Then the hair trigger effect for \eqref{eq-aronson} is naturally linked with the Fujita blow up phenomena for \eqref{eq-fujita}. Hence, in their seminal work, Aronson and Weinberger \cite{Aro-Wei-78} showed  that the hair trigger effect remains valid for \eqref{eq-aronson} as long as $0<p\leq p_F=\frac 2 N$, whereas some (small enough) initial data may lead to extinction, or quenching, when $p>p_F=\frac 2 N$. See also \cite{Xin-93}, \cite{Beb-Li-Li-97}, \cite{Zla-05}.

Based on our Fujita type results for \eqref{eq}, we shall discuss the hair trigger effect for \eqref{eq2}, thus making more precise the balance between the effect of the dispersal tails and the strength of the Allee effect which allows or not  the hair trigger effect. Let us mention that, rather recently, various new results studying the interplay between some heavy tails and an Allee effect have been proved. Let us mention \cite{Zha-Li-Wan-12}, \cite{Mellet}, \cite{Gui-Hua-15}, \cite{Alf-tails-preprint}, \cite{Alf-Cov-preprint}, and the references therein. In those works, the issue is, in different situations permitting propagation, to determine whether invasion is performed at a constant speed or by accelerating.

\section{Assumptions and results}\label{s:results}

Let us first present and discuss the assumptions on the dispersal kernel $J$.  As observed and proved in \cite{Cha-Cha-Ros-06}, expansion \eqref{J-Fourier-ass} plays a crucial role in the behavior of the linear equation $\partial _t u=J*u-u$, and so will for the nonlinear problem \eqref{eq}.  

\begin{assumption}[Dispersal kernel]
\label{ass:J} $J:\R^N\to \R$ is  nonnegative, bounded, radial and satisfies $\int _{\R^{N}}J=1$. Its Fourier transform has an expansion
\begin{equation}
\label{J-Fourier-ass}
\J(\xi)=1-A\vert \xi\vert ^{\beta}+o(\vert \xi \vert ^{\beta}), \quad \text{ as } \xi \to 0,
\end{equation}
for some $0<\beta\leq 2$ and $A>0$.
\end{assumption}


Notice that expansion \eqref{J-Fourier-ass} contains the information on the tails of $J$. Indeed, for kernels which have a finite second momentum, namely $m_2:=\int_{\R ^N}\vert x\vert ^2J(x)dx<+\infty$,  expansion \eqref{J-Fourier-ass} holds true with $\beta =2$, as can be seen in \cite[Chapter 2, subsection 2.3.c, (3.8) Theorem]{Dur-96} among others. In particular, this is the case for kernels which are compactly supported, exponentially bounded, or which decrease like  $\mathcal O\left(\frac{1}{\vert x\vert^{N+2+\ep}}\right)$ with $\ep>0$.

On the other hand, when $m_2=+\infty$ then more general expansions are possible. For example, for algebraic tails satisfying 
\begin{equation}
\label{heavytails}
J(x)\sim \frac{C}{\vert x\vert ^{\alpha}} \quad \text{ as }\vert x\vert \to \infty, \quad\text{ with } N<\alpha< N+2,
\end{equation}
then \eqref{J-Fourier-ass} holds true with $\beta= \alpha- N\in(0,2)$. This fact is related to the stable laws of index $\beta\in(0,2)$ in probability theory, and a proof can be found in \cite[Chapter 2, subsection 2.7]{Dur-96}. In particular it contains the case of the Cauchy law $J(x)=\frac{1/\pi}{1+x^2}$ (when $N=1$), for which
$$
 \widehat J(\xi)=e^{-\vert\xi\vert}=1-\vert\xi\vert +o(\vert \xi\vert), \quad  \text{ as } \xi \to 0,
$$
and $\beta =1$, despite the nonexistence of the first momentum $m_1:=\int \vert x\vert J(x)dx$.

\begin{rem} [Critical algebraic tails] For algebraic tails
$$
J(x)\sim \frac{C}{\vert x\vert ^{N+2}}, \quad \text{ as }\vert x\vert \to \infty,
$$
which are critical for the nonexistence of the second momentum $m_2$, expansion \eqref{J-Fourier-ass} is replaced by
$$
\J(\xi)=1+A\vert \xi\vert ^2 \ln \vert \xi\vert +o(\vert \xi\vert ^2 \ln \vert \xi\vert), \quad \text{ as } \xi \to 0.
$$
Nevertheless, as proved in \cite[Theorem 5.1]{Cha-Cha-Ros-06}, the solutions to $\partial _t u=J*u-u$ still tend to that of the Heat equation, but with a different time velocity. In other words, for such tails, we do believe that $p_F=\frac 2 N$ and that this can be proved by additonal technicalities and by using \cite[Theorem 5.1]{Cha-Cha-Ros-06} rather than \cite[Theorem 1]{Cha-Cha-Ros-06} to derive an anologous of Lemma \ref{lem:decrease}.
\end{rem}

\medskip

Before going further, we need to say a word on the notion of solutions. A function $u\in C^1((0,T),L^\infty(\R ^N)\cap L^1(\R ^N))\cup C^0([0,T),L^\infty(\R ^N)\cap L^1(\R ^N))$ for some $T > 0$, which satisfies the equation a.e. in $(0,T)\times\R^{N}$ is a local solution to \eqref{eq} with $u(0,\cdot)$ as initial data. For such solutions, the comparison principle is available. Furthemore, for a nonnegative $u_0\in L^\infty(\R^N)\cap L^1(\R ^N)$, the associated Cauchy problem \eqref{eq} admits a unique solution defined on some maximal interval $[0,T)$. Moreover either $T=\infty$ and the solution is global, either $T<\infty$ and then both  $\Vert u(t,\cdot)\Vert _{L^1}$ and $\Vert u(t,\cdot)\Vert _{L^\infty}$ tend to $\infty$ as $t\nearrow T$, which is called blow up in finite time. These facts are rather well-known, and parts of them can be found in \cite{Gar-Qui-10} for instance. 

\medskip

As explained in the introduction, our main result is the identification of the Fujita exponent for the nonlocal equation \eqref{eq} for a large class of dispersion kernels, namely those admitting an expansion \eqref{J-Fourier-ass}. In this context, $p_F:=\frac{\beta}{N}$ is the Fujita exponent. More precisely, the following holds.

\begin{theo}[Systematic blow up]
\label{th:systematic}
Let Assumption \ref{ass:J} hold. Assume $0<p\leq p_F=\frac{\beta}{N}$. Assume $u_0\in L^\infty(\R^N)\cap L^1(\R ^N)$ is nonnegative and satisfies --- for some $\ep>0$, $x_0\in \R ^N$, $r>0$--- $u_0(x)\geq \ep$ for all $x\in B(x_0,r)$. Then  the solution to the Cauchy problem \eqref{eq} with $u_0$ as initial data blows up in finite time.
\end{theo}

\begin{theo}[Blow up vs extinction]\label{th:vs} Let Assumption \ref{ass:J} hold. Assume $p>p_F=\frac{\beta}{N}$. Then the following holds.
\begin{enumerate}[(i)]
\item There is $\delta>0$ such that, for any nonnegative  $u_0\in L^\infty(\R^N)\cap L^1(\R ^N)$ with
$$
\Vert u_0\Vert_{L ^{1}} +\Vert \uzero\Vert _{L ^{1}}<\delta,
$$
the solution to the Cauchy problem \eqref{eq} is global in time and  satisfies, for some $C>0$,
\begin{equation}
\label{extinction}
\Vert u(t,\cdot)\Vert _{L^\infty}\leq \frac{C}{(1+t)^{N/\beta}}, \quad \text{ for any } t\geq 0.
\end{equation}

\item On the other hand, assume $\lambda >0$ and $R>0$ are such that
\begin{equation}
\label{lambda-R}
\lambda > \left(1-C_N \int_{\vert z\vert \leq R} J(z)\, dz\right)^{1/p},
\end{equation}
where $0<C_N<1$ is a constant that depends only on the dimension $N$ (see subsection \ref{ss:large} for the exact value of this constant). Then, the solution to the Cauchy problem  \eqref{eq} with $\lambda \mathbf 1_ {\{\vert x\vert \leq R\}}$ as initial data blows up in finite time.
\end{enumerate}
\end{theo}

As regards condition \eqref{lambda-R}, let us notice that if $\lambda >1$ then it is satisfied for any $R>0$, indicating that large $L^\infty$ data always lead to blow up; if $(1-C_N)^{1/p}<\lambda\leq 1$ then \eqref{lambda-R} is satisfied by taking $R>0$ sufficienty large, indicating that intermediate  $L^\infty$ data require large initial mass to blow up (at least in our result); if $\lambda\leq (1-C_N)^{1/p}$ then \eqref{lambda-R} is never satisfied, indicating that small $L^\infty$ data are bad candidates for blowing up.

Let us recall that, when $J$ is compactly supported, the fact that the Fujita exponent $p_F=\frac 2 N$ is the same as that of the nonlinear diffusion equation \eqref{eq-fujita} was already proved in \cite{Gar-Qui-10}. Nevertheless, our results assert further that this remains true for kernels $J$ which have a finite second momentum. On the other hand, when $\beta <2$ in expansion \eqref{J-Fourier-ass} the Fujita exponent becomes that of the Fractional equation \eqref{eq-fujita} with $s=\beta \in(0,2)$. Hence, depending on the tails of the dispersal kernel, the nonlocal equation \eqref{eq}  behaves with respect to blow up either like the local Heat equation \eqref{eq-fujita}, either like a Fractional equation \eqref{eq-sugitani}. This sheds light on the richness of \eqref{eq}.


\medskip

We now turn to the hair trigger effect for \eqref{eq2}, whose analysis makes use of the Fujita type results for \eqref{eq}. Notice that if $0\leq u_0\leq \Vert u_0\Vert _\infty<+\infty$ then,  from the comparison principle, we get that the solution to \eqref{eq2} satisfies
$$
0<u(t,x)\leq \max(1,\Vert u_0\Vert _\infty), \quad \forall (t,x) \in(0,\infty)\times \R ^N,
$$
so that the solution is always global.

\begin{cor}[Hair trigger effect along a subsequence vs. quenching solutions]
\label{cor:hair}
Let Assumption \ref{ass:J} hold.
\begin{enumerate}[(i)]
\item Assume $0<p\leq p_F=\frac{\beta}{N}$. Assume $u_0:\R^N\to [0,1]$ is continuous and non trivial. Then  the solution to the Cauchy problem \eqref{eq2} with $u_0$ as initial data satisfies
\begin{equation}
\label{hair}
\limsup _{t\to \infty} \inf _{\vert x\vert \leq R} u(t,x)=1, \quad \text{ for any } R\geq 0.  
\end{equation}
\item Assume $p>p_F=\frac{\beta}{N}$. Then there is $\delta>0$ such that, for any nonnegative, continuous, bounded and integrable $u_0$ with
$
\Vert u_0\Vert_{L ^{1}} +\Vert \uzero\Vert _{L ^{1}}<\delta$, 
the solution to the Cauchy problem \eqref{eq2} satisfies, for some $C>0$,
\begin{equation}
\label{extinction2}
\Vert u(t,\cdot)\Vert _{L^\infty}\leq \frac{C}{(1+t)^{N/\beta}}, \quad \text{ for any } t\geq 0.
\end{equation}
\end{enumerate}
\end{cor}

Observe that Corollary \ref{cor:hair} $(ii)$ directly follows from Theorem \ref{th:vs} $(i)$ and the comparison principle.  Corollary $(i)$, whose proof is rather classical,  is the hair trigger effect, but only along a subsequence of time. Under a more restrictive assumption on the exponent  $p$, we can actually prove the following hair trigger effect.

\begin{theo}[Hair trigger effect]
\label{th:hair}
Let Assumption \ref{ass:J} hold. Assume $0<p<\frac 12  p_F=\frac 12 \frac{\beta}{N}$. Assume $u_0:\R^N\to [0,1]$ is continuous and non trivial. Then  the solution to the Cauchy problem \eqref{eq2} with $u_0$ as initial data satisfies
\begin{equation}
\label{hair2}
\lim _{t\to \infty} \inf _{\vert x\vert \leq R} u(t,x)=1, \quad \text{ for any } R\geq 0.  
\end{equation}
\end{theo}

The proof of the above result requires the combination of an elaborate subsolution and careful asymptotics of the solution to the linear nonlocal diffusion equation $\partial _t u=J*u-u$. Using such a strategy, it seems very difficult, if possible, to remove the assumption $0<p<\frac 12 p_F$. Hence, different approaches should be used for the range $\frac 12 p_F \leq p\leq p_F$, where more complex scenarios may exist.  We hope to address this issue in a future work.

Notice also that, these results remain valid for equation
$$
\partial _t u=J*u-u+f(u),
$$
as long as $f$ satisfies, for instance, $f(u)\sim r u^{1+p}$ as $u\to 0$ (for some $r>0$), $f>0$ on $(0,1)$, $f(1)=0$, $f'(1)<0$, $f<0$ on $(1,\infty)$. Indeed, in such a case, we can sandwich $mu^{1+p}(1-u)\leq f(u)\leq M u^{1+p}(1-u)$ for some $m>0$, $M>0$, and then combine some comparison and rescaling arguments.

\medskip

The paper  is organized as follows. We recall basic facts in Section \ref{s:basic}. In Section \ref{s:blowup}, we  prove the systematic blow up of any solution when $0<p\leq p_F$, that is Theorem \ref{th:systematic}. We study the case $p>p_F$ in Section \ref{s:extinction}, proving blow up or extinction depending on the size of the initial data, as stated in Theorem \ref{th:vs}. Last, in Section \ref{s:hair-trigger}, we prove the hair trigger effect, as stated in Corollary \ref{cor:hair} $(i)$ and Theorem \ref{th:hair}.

\section{Notations and basic facts}\label{s:basic}

Before proving our results, let us now introduce some notations and recall briefly some basic facts.
\medskip

For any integrable function $J$, we define
$$
K(t,\cdot):=e^{-t}\dirac  + e^{-t}\sum_{k=1
}^{+\infty}\frac{t^k}{k!}J^{*(k)}=:e^{-t}\dirac  +\psi(t,\cdot)\,,
$$
where $\dirac$ is the Dirac mass at $0$ and
$J^{*(k)}:=J*\cdots*J$ is the convolution
of $J$ with itself $k-1$ times.

Then, the (unique)
bounded solution to $\partial_t u = J*u-u$ with initial
condition $u_0\in L^\infty(\R^N)$ is given by
$$
u(t,x)=K(t,\cdot)*u_0 (x)=e^{-t}u_0(x)+\psi(t,\cdot)*u_0 (x).
$$
Obviously, though the convolution
of a Dirac mass by an $L^\infty$ function is not pointwise well
defined, we let $\dirac*u_0=u_0$. Also, from the normal convergence,
in $C([0,T];L^1(\R^N))$, of the series
$\sum_{k= 1}^{+\infty}\frac{t^k}{k!} J^{*(k)}$ and $\sum_{k= 1}^{+\infty}\frac{t^{k-1}}{(k-1)!} J^{*(k)}$ we deduce that the function $t\in [0,\infty)\mapsto \psi(t,\cdot)\in L^{1}(\R ^N)$ is of class $C^{1}$ and that
\begin{equation}
\label{eq:psi}
\partial _t \psi(t,x)=J*\psi(t,\cdot)(x)-\psi(t,x)+e^{-t}J(x).
\end{equation}
Notice also that
\begin{equation}
\label{mass-presque-1}
\int _{\R ^N}\psi(t,x)dx=1-e^{-t}.
\end{equation}

\medskip

 For the sake of clarity, let us state our conventions on the Fourier transform. If $f\in L^1(\R ^N)$, we define its Fourier transform $\mathcal F (f)=\widehat f$ and its inverse Fourier transform $\mathcal F^{-1}(f)$ by
$$
\widehat f(\xi):=\int _{\R ^N}e^{-i\xi \cdot x}f(x)dx, \quad \mathcal F^{-1}(f)(x):=\int _{\R ^N}e^{ix\cdot \xi }f(\xi)d\xi.
$$
With this definition, we have, for $f$, $g\in L^1(\R ^N)$,
$$
\widehat{f*g}=\widehat f \, \widehat g,
$$
and $f=\frac{1}{(2\pi)^{N}}\mathcal F ^{-1}(\mathcal F(f))$ if $f$, $\mathcal F (f) \in L ^1(\R ^N)$. Also, after defining the Fourier transform on $L^2(\R^N)$ we get the Plancherel formula
$$
\int _{\R ^N} f(x)g(x)dx=\frac{1}{(2\pi)^N}\int _{\R ^N} \widehat f(\xi)\widehat g(\xi)d\xi,
$$
for $f$, $g \in L^2(\R ^N)$.
\section{Systematic blow up}\label{s:blowup}

In this section, we first provide a priori estimates on a crucial quantity related to possible global solutions of \eqref{eq}. They will then enable us to prove the blow up of any solution when $0<p\leq p_F$, as stated in Theorem \ref{th:systematic}.

\subsection{Some a priori estimates}\label{ss:crucial}

In this subsection, we assume that $u_0$ is nonnegative, nontrivial, radial, continuous, bounded, and that both $u_0$ and $\uzero$ are in $L^{1}(\R ^{N})$. We also assume that we are equipped with a global solution $u(t,x)$ of the associated Cauchy problem \eqref{eq}.  We then define, for any $t\geq 0$, the c quantity
\begin{equation}
\label{def:f}
f(t):=\int _{\R ^N} \E \uzero (\xi) d\xi.
\end{equation}
In the spirit of an original idea of Kaplan \cite{Kap-63}, also used in \cite{Fuj-66}, we are going to estimate $f(t)$ from below and above as $t\to \infty$. As clear in the following, another key ingredient is the Fourier duality which enables to recast \eqref{def:f} as \eqref{def:f-dual}.

\begin{lem}[Estimate from below]
\label{lem:f-below}
There is a constant $G>0$  depending only on the dimension $N$ and the kernel $J$, and a constant $t_0>0$ (that is allowed to depend on the initial data), such that
\begin{equation}
\label{f-below}
f(t)\geq \frac{G\Vert u_0\Vert _{L^1}}{t^{N/\beta}} \quad \text{ for any } t\geq t_0 .
\end{equation}
\end{lem}

\begin{proof}
From \eqref{J-Fourier-ass}, we can  select $\xi _0>0$ small enough so that 
\begin{equation}
\label{preszero}
\vert \xi\vert\leq \xi _0 \Longrightarrow \J(\xi)-1\geq -2A\vert \xi \vert ^{\beta}.
\end{equation}
Since $\uzero (0)=\int_{\R ^N} u_0 >0$ and $\uzero$ is a real valued continuous function, up to reducing $\xi _0>0$ if necessary, we can assume that
$$
\vert \xi\vert\leq \xi _0 \Longrightarrow \uzero(\xi)\geq 0.
$$
On the other hand, $\J$ is continuous, $\J(\xi)-1<0$ for all $\xi\neq 0$, $\J(\xi)-1\to -1$ as $\vert \xi \vert \to +\infty$, hence there is $\delta >0$ such that
\begin{equation}
\label{loinzero}
\vert \xi\vert\geq \xi _0 \Longrightarrow \J(\xi)-1\leq -\delta.
\end{equation}

Now, by cutting into two pieces, we get $t^{N/\beta}f(t)=g_1(t)+g_2(t)$, where
$$
\vert g_2(t)\vert \leq \left\vert \int _{\vert \xi\vert \geq \xi _0} t^{N/\beta}\E \uzero(\xi)d\xi\right\vert\leq t^{N/\beta} e^{-\delta t} \Vert \uzero \Vert_{L^{1}}\to 0\quad \text{ as } t \to \infty,
$$
and
\begin{eqnarray*}
g_1(t)&=&t^{N/\beta}\int_{\vert \xi\vert \leq \xi_0} \E \uzero(\xi)d\xi\\
&\geq& t^{N/\beta}\int_{\vert \xi\vert \leq \xi_0}e^{-2At\vert \xi\vert^\beta} \uzero(\xi)d\xi\\
&=&\int _{\R^{N}} e^{-2A\vert z\vert ^{\beta}}\uzero\left(\frac{z}{t^{1/\beta}}\right)\mathbf 1 _{(0,t^{1/\beta}\xi _0)}(\vert z\vert)dz.
\end{eqnarray*}
By the dominated convergence theorem, the last integral above tends, as $t\to \infty$, to the constant 
$$
\uzero(0)\int _{\R ^N} e^{-2A\vert z\vert ^\beta}dz=\Vert u_0\Vert _{L^1}\int _{\R ^N} e^{-2A\vert z\vert ^\beta}dz=:\Vert u_0\Vert _{L^1}2G,
$$
where $G>0$ depends only on the dimension $N$ and the kernel $J$ (via $A$ and $\beta$). As a result, we can select $t_0>0$ large enough so that \eqref{f-below} holds true. The lemma is proved.
\end{proof}


In order to derive an estimate from above, it is more convenient to use the  dual expression (see below for a proof)
\begin{equation}
\label{def:f-dual}
f(t)= (2\pi)^N \int_{\R ^{N}}e^{-t}\left(\delta _0+\sum _{k=1}^{+\infty}\frac{t^{k}}{k!}J^{*(k)}(x)\right)u_0(x)dx= (2\pi)^N \int_{\R ^{N}}K(t,x)u_0(x)dx,
\end{equation}
where we recall that $K(t,x)$ was defined in Section \ref{s:basic}. Notice that, formally, the fundamental solution of $\partial _tu=J*u-u$ is $\mathcal F ^{-1}(\E)=e^{-t}\left(\delta _0+\sum _{k=1}^{+\infty}\frac{t^{k}}{k!}J^{*(k)}(x)\right)$ so that expression \eqref{def:f-dual} is, again formally, derived from \eqref{def:f} by the Plancherel formula. 

\begin{proof}[Proof of \eqref{def:f-dual}] From \eqref{def:f} we get
$$
e^tf(t)-\int _{\R ^N}\uzero(\xi)d\xi=\int_{\R ^N}\sum_{k=1}^{+\infty}\frac{t^k}{k!}\widehat J \,^k(\xi) \uzero (\xi)d\xi
=\sum _{k=1}^{+\infty} \int _{\R ^N} \frac{t^k}{k!}\widehat J\, ^k(\xi) \uzero (\xi)d\xi,
$$
since  $\sum _{k} \int \vert \frac{t^k}{k!}\widehat J \,^k(\xi) \uzero (\xi)\vert d\xi\leq \sum _k \frac{t^{k}}{k!}\Vert \uzero\Vert _{L^{1}}<+\infty$ (recall that $\vert \J(\xi)\vert \leq 1$). Next $J\in L^{1}(\R ^{N})$ implies $\widehat J \, ^{k}(\xi)=\widehat{J^{*(k)}}(\xi)$, so that
$$
e^tf(t)-\int _{\R ^N}\uzero(\xi)d\xi
=\sum _{k=1}^{+\infty} \int _{\R ^N} \frac{t^k}{k!}\widehat{J^{*(k)}}(\xi) \uzero (\xi)d\xi.
$$
Next, both $J^{*(k)}$, $u_0 \in L^{\infty}(\R ^{N})\cap L^{1}(\R ^{N})\subset L^{2}(\R ^{N})$ so that we can apply the Plancherel formula to get
\begin{eqnarray}
e^tf(t)-\int _{\R ^N}\uzero(\xi)d\xi&=&(2\pi)^{N}\sum _{k=1}^{+\infty} \int _{\R ^N} \frac{t^k}{k!} J ^{*(k)}(x) u_0(x)dx\nonumber \\
&=&(2\pi)^{N} \int _{\R ^N} \sum _{k=1}^{+\infty}\frac{t^k}{k!} J ^{*(k)}(x) u_0(x)dx,\label{plug}
\end{eqnarray}
 since $\sum _{k} \int \vert \frac{t^k}{k!}J ^{*(k)}(x) u_0 (x)\vert dx\leq \sum _k \frac{t^{k}}{k!} \Vert J\Vert _{L^\infty}\Vert u_0 \Vert _{L^{1}}<+\infty$. Last, notice that $\int_{\R ^{N}} \uzero =\widehat \uzero (0)$. Since $u_0$ is real and radial, so is $\uzero$, which in turn implies $\widehat \uzero=\mathcal F^{-1}( \uzero)=(2\pi)^{N}u_0$ so that 
 $\int_{\R ^{N}} \uzero=(2\pi)^{N}u_0(0)=
 (2\pi)^{N}\int_{\R ^{N}}\delta _0 u_0$, which we plug into \eqref{plug} to  conclude the proof of \eqref{def:f-dual}.
 \end{proof}
 
\medskip

Equipped with the dual formula \eqref{def:f-dual}, we can now prove the following.

\begin{lem}[Estimate from above] 
\label{lem:f-above}
We have
\begin{equation}
\label{f-above}
f(t)\leq (2\pi)^N \left(\left(\frac{p+1}{p}\right)^{1/p}\frac{1}{t^{1/p}}+e^{-t}u_0(0)\right) \quad \text{ for any } t>0.
\end{equation}
\end{lem}

\begin{proof} Let $T>0$ be given. Denote $C_T:=\max _{0\leq \tau\leq T+1} \Vert u(\tau,\cdot)\Vert _{L^{\infty}} +\Vert u(\tau,\cdot)\Vert _{L^{1}}<+\infty$.
Fix some $0< t\leq T$. 

First observe that \eqref{def:f-dual} is recast 
\begin{equation}
\label{f-dual-psi}
h(t):=\frac{f(t)}{(2\pi)^{N}}=\int _{\R^{N}}\left(e^{-t}\delta _0+\psi(t,x)\right)u_0(x)dx=e^{-t}u_0(0)+\int_{\R ^N} \psi(t,x)u_0(x)dx,
\end{equation}
where $\psi(t,x)=e^{-t}\sum_{k=1
}^{+\infty}\frac{t^k}{k!}J^{*(k)}(x)$ is as in Section \ref{s:basic}.

For $0<\ep\leq 1$, let us define
\begin{equation}
\label{defgep}
g_\ep(s):=\int _{\R^{N}}\psi(t-s+\ep,x)u(s,x)dx, \quad 0\leq s \leq t.
\end{equation}
Notice that, if $k$ is sufficiently large, the support of $J^{*(k)}$ meets that of $u(s,\cdot)$, and therefore $\int_{\R ^{N}}J ^{*(k)}(x)u(s,x)dx>0$, which in turn implies  $g_\ep(s)>0$. Notice also that, using the dominated convergence theorem, we see that, as $\ep \to 0$,
\begin{equation}
\label{eptozero}
g_\ep(0)=\int _{\R ^N}\psi(t+\ep,x)u_0(x)dx \to \int_{\R ^N}\psi(t,x)u_0(x)dx=\frac{f(t)}{(2\pi)^N}-e^{-t}u_0(0).
\end{equation}

Using the constant $C_T$ defined above one can dominate the partial derivative with respect to $s$ of the integrand in \eqref{defgep}, and therefore prove that $g_\ep$ is differentiable. Using equations  \eqref{eq:psi} and \eqref{eq}, we then compute
\begin{eqnarray}
g'_\ep(s)&=&\int _{\R^{N}}(-J*\psi(t-s+\ep,\cdot)+\psi(t-s+\ep,\cdot)-e^{-(t-s+\ep)}J)u(s,\cdot)\nonumber\\
&&+\int _{\R^{N}}\psi(t-s+\ep,\cdot)(J*u(s,\cdot)-u(s,\cdot)+u^{1+p}(s,\cdot))\nonumber\\
&=&-\int _{\R^{N}} e^{-(t-s+\ep)}J u(s,\cdot)+\int _{\R^{N}} \psi(t-s+\ep,\cdot)u^{1+p}(s,\cdot),\label{gprime}
\end{eqnarray}
by Fubini theorem. From the expression of $\psi$, we see that $e^{-\tau}J(x)\leq \frac{\psi(\tau,x)}{\tau}$ for $\tau>0$, so that\begin{equation}
\label{premier}
\int _{\R^{N}} e^{-(t-s+\ep)}J u(s,\cdot)\leq \frac{1}{t-s+\ep}g_\ep(s).
\end{equation}
Next, we write
\begin{eqnarray}
\int _{\R^{N}} \psi(t-s+\ep,\cdot)u^{1+p}(s,\cdot)&=&(1-e^{-(t-s+\ep)})\int _{\R^{N}}\frac{\psi(t-s+\ep,\cdot)}{1-e^{-(t-s+\ep)}}u^{1+p}(s,\cdot)\nonumber \\
&\geq & \frac{1}{(1-e^{-(t-s+\ep)})^{p}}g_\ep^{1+p}(s)\nonumber\\
&\geq & g_\ep ^{1+p}(s),\label{second}
\end{eqnarray}
where we have used the Jensen inequality (notice that $\int _{\R^{N}}\frac{\psi(t-s+\ep,\cdot)}{1-e^{-(t-s+\ep)}}=1$ in view of \eqref{mass-presque-1}). Plugging \eqref{premier} and \eqref{second} into \eqref{gprime} and multiplying by the 
integrating factor
$(t-s+\ep)^{p}$ we arrive at
$$
\left(\frac{g_\ep'(s)}{g_\ep ^{1+p}(s)}+\frac{1}{t-s+\ep}\frac{1}{g_\ep ^{p}(s)}\right)(t-s+\ep)^{p}\geq (t-s+\ep)^{p}.
$$
The left hand side member is nothing else that $\frac{d}{ds}\left(\frac{(t-s+\ep)^p}{-pg_\ep^p(s)}\right)$ so that
integrating from $0$ to $t$, we get
$$
-\frac{1}{p}\frac{1}{g_\ep ^{p}(t)}\ep ^{p}+\frac 1 p\frac{1}{g_\ep ^{p}(0)}(t+\ep)^{p}\geq -\frac{\ep^{p+1}}{p+1}+\frac{(t+\ep) ^{p+1}}{p+1},
$$
which in turn implies
$$
\frac{1}{g_\ep ^{p}(0)}\geq \frac{p}{p+1}(t+\ep)-\frac p{p+1} \frac{\ep ^{p+1}}{(t+\ep)^{p}}.
$$
Letting $\ep\to 0$ and using \eqref{eptozero}, we get estimate \eqref{f-above}, which concludes the proof of Lemma \ref{lem:f-above}.
\end{proof}

\subsection{Proof of systematic blow up}\label{ss:blowup}

\begin{proof}[Proof of Theorem \ref{th:systematic}] Let $u_0\in L^\infty(\R^N)$ be nonnegative and satisfying --- for some $\ep>0$, $x_0\in \R ^N$, $r>0$--- 
$u_0(x)\geq \ep$ for all $x\in B(x_0,r)$. Thus there exists a nonnegative, nontrivial, radial and $C^{\infty}_c(\R ^N)$ function that is smaller than $u_0$. By the comparison principle, it is enough to prove blow up for such an initial data. Hence we can assume without loss of generality that $u_0$ is nonnegative, nontrivial, bounded, that $u_0\in C^{\infty}_c(\R ^N)$, and thus $\uzero\in\mathcal S(\R^{N})$. Hence, assuming by contradiction existence of a global solution,  all the results of subsection \ref{ss:crucial} are available. 

$\bullet$ When $0<p<\frac \beta N$, letting $t\to \infty$ in \eqref{f-below} and \eqref{f-above} immediatley gives a contradiction.

$\bullet$ In the critical case $p=\frac \beta N$, letting $t\to \infty$ in \eqref{f-below} and \eqref{f-above} only provides $\Vert u_0\Vert _{L^{1}}\leq C$, where the constant $C>0$ depends on the dimension $N$ and the kernel $J$ but not on the size of the initial data. Thus, by regarding  $u(t,\cdot)$ as an initial value, we derive that
\begin{equation}
\label{norme-L1}
m(t):=\Vert u(t,\cdot)\Vert _{L^{1}} \leq C, \quad \text{ for any } t\geq 0.
\end{equation}

Integrating equation \eqref{eq} over $x\in \R ^N$ and using Fubini theorem, we get
$$
\frac{d}{dt}m(t)=\int_{\R ^N} u^{1+p}(t,x)\,dx,
$$
so that $\int _0^{t}\int _{\R ^N} u^{1+p}(t,x)\,dxdt= m(t)-m(0)\leq C$, for all $t\geq 0$. As a result we know that
\begin{equation}
\label{int-double}
\int _0 ^{\infty} \int _{\R  ^N} u^{1+p}(t,x)\,dxdt<+\infty.
\end{equation}

We are going to derive below a contradiction, using a modification of an original technic of \cite{Mit-Pok-01} for a local equation. Notice that our kernel $J$ may not have finite second nor first moment, so we need to  derive further estimates. Also, we shall again take advantage of the Fourier duality.

Consider $\rho \in C_c^\infty(\R)$ such that $\rho \equiv 1$ on $(-1,1)$, $0\leq \rho\leq 1$ and $\text{Supp } \rho=[-2,2]$. Let $T>0$ be given. Let $\ep>0$ be given. For $R>0$, we define
$$
\psi_R(t):=\rho\left(\frac{t-T}{R^{Np}}\right)=\rho\left(\frac{t-T}{R^{\beta}}\right), \quad \theta _R(x):=\rho\left(\ep \frac{\vert x\vert}{R}\right).
$$
We multiply equation \eqref{eq} by $\theta _R(x)\psi _R(t)$ and integrate over $(t,x)\in (T,\infty)\times \R ^N$ to get
\begin{eqnarray}
\int _T^{\infty}\int_{\R^{N}} u^{1+p}(t,x)\theta _R(x)\psi_R(t) &=& -\int _T^{\infty}\int_{\R^{N}} (J*u-u)(t,x)\theta _R(x)\psi_R(t)\nonumber\\
&&+\int _T^{\infty}\int_{\R^{N}}  \partial _t u(t,x) \theta _R(x)\psi_R(t)\nonumber\\
&\leq & -\int _T^{\infty}\int_{\R^{N}} (J*\theta _R-\theta _R)(x)u(t,x)\psi_R(t)\nonumber\\
&& -\int _T^{\infty}\int_{\R^{N}}   u(t,x) \theta _R(x)\psi_R'(t)=:-I_1-I_2,\label{ineg}
\end{eqnarray}
where we have used Fubini theorem, integration by part in time, in the first, respectively the second, integral of the right hand side member. In the sequel, we denote by $C$ a positive constant that may change from place to place but that is always independent on  $\ep>0$ and $R>0$.

Let us deal with $I_2=I_2(R)$.  Observe that
$$
\vert \psi_R'(t)\vert =\vert \frac{1}{R^{\beta}}\rho '\left(\frac{t-T}{R^{\beta}}\right)\vert \leq \frac{C}{R^{\beta}}\mathbf{1}_{(T+R^{\beta},T+2R^{\beta})}(t),
$$
so that
\begin{eqnarray}
\vert I_2\vert &\leq& \frac{C}{R^{\beta}}\int _{T+R^{\beta}}^{T+2R^{\beta}}\int_{\vert x\vert \leq 2R/\ep }   u(t,x)\label{footstep}\\
&\leq &  \frac{C}{R^{\beta}}\left(\int _{T+R^{\beta}}^{T+2R^{\beta}}\int_{\vert x\vert \leq 2R/\ep } 1\right) ^{\frac{p}{p+1}}\left(\int _{T+R^{\beta}}^{T+2R^{\beta}}\int_{\vert x\vert \leq 2R/\ep } u^{1+p}(t,x)\right)^{\frac{1}{p+1}}\nonumber\\
&=& \frac{C}{R^{\beta}} \left(
R^{\beta}
\left(
\frac{2R}{\ep}
\right)^{N}
\right)^{\frac{p}{p+1}}\left(\int _{T+R^{\beta}}^{T+2R^{\beta}}\int_{\vert x\vert \leq 2R/\ep } u^{1+p}(t,x)\right)^{\frac{1}{p+1}}\nonumber\\
&=&\frac C{\ep^{\frac{Np}{p+1}}} \left(\int _{T+R^{\beta}}^{T+2R^{\beta}}\int_{\vert x\vert \leq 2R/\ep } u^{1+p}(t,x)\right)^{\frac{1}{p+1}},\nonumber
\end{eqnarray}
where we have used the H\"older inequality and equality $\beta=Np$. In view of \eqref{int-double}, the last integral above tends to zero as $R\to \infty$, and so does $I_2$.

Let us deal with $I_1=I_1(R)=\int_T^{T+2R^{\beta}}\int_{\R ^N} B(x)u(t,x)$, where
$$
B(x):=(J*\theta _R  -\theta _R)(x)=\int_{\R^N}\left(\theta_R( z-x)-\theta _R(x)\right)J(z)dz.
$$
First observe that if $\vert x\vert \geq 2R/\ep$ then $\theta_R(x)=0$ so that $B(x)\geq 0$. As a result
\begin{equation}
\label{I1prime}  
I_1\geq I_1':=\int_T^{T+2R^{\beta}}\int_{\vert x\vert <2R/\ep} B(x)u(t,x).
\end{equation}
In order to estimate $B(x)$, we use the 
 Plancherel formula and get
\begin{eqnarray*}
(2\pi)^{N}B(x)&=& (2\pi)^{N}\int_{\R^N}\theta_R( z-x)J(z)dz-(2\pi)^{N}\theta_R(x)\\
&=&\int _{\R ^N} \J(\xi)e^{-ix\cdot \xi}\widehat{\theta  _R}(\xi)\,d\xi-(2\pi)^{N}\theta_R(x)\\
&=&\int _{\R ^N} (1-\mathcal A(\xi)\vert \xi\vert ^{\beta})e^{-ix\cdot \xi}\widehat{\theta  _R}(\xi)\,d\xi-(2\pi)^{N}\theta_R(x),
\end{eqnarray*}
where function $\mathcal A$ is bounded in view of \eqref{J-Fourier-ass}. Since $\int _{\R ^N} e^{-ix\cdot \xi}\widehat{\theta  _R}(\xi)\,d\xi=\mathcal F(\mathcal F(\theta _R))(x)=(2\pi)^{N}\theta_R(x)$ and $\widehat{\theta_R}(\xi)=\left(\frac R\ep \right)^{N}\widehat \rho (\frac R\ep \xi)$, we get
$$
(2\pi)^{N}B(x)=-\left(\frac{\ep}{R}\right)^{\beta}\int_{\R ^N} \mathcal A\left(\frac{\ep}{R}\xi'\right)\vert\xi'\vert ^{\beta}e^{-i\frac{\ep}{R}x\cdot\xi'}\widehat \rho (\xi')\,d\xi',
$$
so that $$
\vert B(x)\vert \leq \frac 1{(2\pi)^{N}} \left(\frac{\ep}{R}\right)^{\beta}\Vert \mathcal A\Vert _\infty \int_{\R ^{N}} \vert\xi'\vert ^{\beta}\vert \widehat \rho (\xi')\vert\,d\xi'=C\frac{\ep^{\beta}}{R^{\beta}},
$$
 since $\widehat \rho \in \mathcal S(\R ^N)$. As a result
 $$
 \vert I_{1}'\vert \leq C \frac{\ep^{\beta}}{R^{\beta}}\int _T^{T+2R^{\beta}}\int_{\vert x\vert <2R/\ep}u(t,x).
 $$
We are now in the footsteps of \eqref{footstep} so that --- notice the presence of the crucial  multiplicative factor $\ep ^{\beta}$---  similar arguments (H\"older inequality and $\beta=Np$) yield
\begin{equation}
\label{I1prime-bis}
\vert I_{1}'\vert \leq  C{\ep^{\frac{\beta p}{p+1}}} \left(\int _{T}^{T+2R^{\beta}}\int_{\vert x\vert \leq 2 R/\ep } u^{1+p}(t,x)\right)^{\frac{1}{p+1}}.
\end{equation}

To conclude, plugging \eqref{I1prime} and \eqref{I1prime-bis} into \eqref{ineg}, we get
$$
\int _T^{\infty}\int_{\R^{N}} u^{1+p}(t,x)\theta _R(x)\psi_R(t)\leq \vert I_2\vert +C{\ep^{\frac{\beta p}{p+1}}} \left(\int _{T}^{T+2R^{\beta}}\int_{\vert x\vert \leq 2 R/\ep } u^{1+p}(t,x)\right)^{\frac{1}{p+1}}.
$$
Letting $R\to \infty$ yields
$$
\int _T ^{\infty} \int _{\R^{N}} u^{1+p}\leq C\ep^{\frac{\beta p}{p+1}} \left(\int _{T}^{\infty}\int_{\R ^{N}} u^{1+p}\right)^{\frac{1}{p+1}}.
$$
From the arbitrariness of $\ep >0$ and $T>0$ we deduce that $u\equiv 0$ on $(0,\infty)\times \R^{N}$, which is a contradiction. This concludes the proof of Theorem \ref{th:systematic}.
 \end{proof}



\section{Blow up vs extinction}\label{s:extinction}

In this section, we prove that when $p>p_F=\frac \beta N$, depending on the size of the initial data, the solution to the Cauchy problem \eqref{eq} can be global and extincting, or blowing up in finite time, as stated in Theorem \ref{th:vs}.

\subsection{Extinction for small initial data}\label{ss:extinction}

\begin{proof}
[Proof of Theorem \ref{th:vs} $(i)$] The proof, as that in  \cite{Gar-Qui-10}, relies strongly on \cite{Cha-Cha-Ros-06} which provides the rate of decrease of the $L^{\infty}$ norm of the solution of the nonlocal linear equation $\partial _t v=J*v-v$.

\begin{lem}[See Theorem 1 in \cite{Cha-Cha-Ros-06} and Theorem 5 in \cite{Gar-Qui-10}]\label{lem:decrease} There is $C>0$ such that, for any initial data $v_0\in L^1(\R ^N)$ such that $\widehat{v_0} \in L^1(\R ^N)$, the solution of the Cauchy problem $\partial _t v=J*v-v$ satisfies
$$
\Vert v(t,\cdot)\Vert _{L^\infty}\leq \frac{C(\Vert v_0\Vert _{L^1}+\Vert \widehat{v_0}\Vert _{L^1})}{(1+t)^{N/\beta}}, \quad \text{ for any } t\geq 0.
$$
\end{lem}

We look after a supersolution to \eqref{eq} in the form $g(t)v(t,x)$, where $g(t)>0$ is to be determined (with $g(0)=1$) and $v(t,x)$ is the solution of $\partial _t v=J*v-v$ with $u_0$ as initial data. A straighforward computation shows that it is enough to have $\frac{g'(t)}{g^{1+p}(t)}\geq \Vert v(t,\cdot)\Vert _{L^\infty}^{p}$. By the above lemma, it is therefore enough to have
$$
\frac{g'(t)}{g^{1+p}(t)}=\frac{C^{p}(\Vert u_0\Vert _{L^1}+\Vert \uzero\Vert _{L^1}) ^{p}}{(1+t)^{pN/\beta}}, \quad g(0)=1.
$$
If $\Vert u_0\Vert _{L^1}+\Vert \uzero\Vert _{L^1}<\delta:=\frac 1C\left(\frac{\frac{pN}{\beta}-1}{p}\right)^{1/p}$ (notice that $\frac{pN}{\beta}-1>0$) then the solution of the above Cauchy problem 
$$
g(t)=\frac{1}
{\left(
1-\frac{pC^{p}(\Vert u_0\Vert _{L^1}+\Vert \uzero\Vert _{L^1})^{p}}
{\frac{pN}{\beta}-1}
\left(1-
\frac{1}{(1+t)^{\frac{pN}{\beta}-1}}
\right)
\right)^{1/p}},
$$
exists for all $t\geq 0$ and is decreasing. It therefore follows from the comparison principle that $u(t,x)\leq g(t)v(t,x)\leq v(t,x)$ so that the solution $u(t,x)$ of \eqref{eq} is global in time and, in view of Lemma \ref{lem:decrease}, satisfies estimate \eqref{extinction}. This concludes the proof of Theorem \ref{th:vs} $(i)$.
\end{proof}

\subsection{Blow up for large initial data}\label{ss:large}

\begin{proof}
[Proof of Theorem \ref{th:vs} $(ii)$] Let $\lambda >0$ and $R>0$ be given such that \eqref{lambda-R} holds. Now, let us consider the solution $u(t,x)$ to \eqref{eq} with initial data $u_0=\lambda \mathbf 1_{\{\vert x\vert \leq R\}}$ and prove the blow up of the \lq\lq localized mass''
\begin{equation}
\label{def:masse}
m(t):=\int_{\vert x\vert \leq R}u(t,x)dx,
\end{equation}
which is enough to prove the blow up of the solution.

Integrating equation \eqref{eq}, we get
\begin{equation}\label{edo}
\frac{d}{dt}m(t)=\int _{\vert x\vert \leq R} J*u(t,\cdot)(x) \, dx-m(t)+\int_{\vert x\vert \leq R}  u ^{1+p}(t,x)\, dx.
\end{equation}
Denoting $B_N$ the volume of the unit ball in $\R^{N}$, we estimate the last term in the above right hand side member by
\begin{equation}\label{esti1}
\int_{\vert x\vert \leq R}u ^{1+p}(t,\cdot)=B_N R^{N}\int_{\vert x\vert \leq R}\frac{1}{B_N R^{N}}u^{1+p}(t,\cdot)\geq \frac{1}{B_N^{p}R^{Np}}m^{1+p}(t),
\end{equation}
thanks to the Jensen inequality. Let us now turn to the first term in the right hand side member of \eqref{edo}. Using Fubini theorem yields
\begin{eqnarray*}
\int_{\vert x\vert \leq R} J*u(t,\cdot)&=&\int_{\R ^{N}}u(t,y) \int _{\vert x\vert \leq R} J(x-y)\, dx dy\\
&\geq & \int_{\vert y\vert \leq R }u(t,y) \int _{\vert z-y \vert \leq R} J(z)\, dz dy.
\end{eqnarray*}
Now we claim that, for any $y$ such that $0<\vert y\vert < R$,
\begin{equation}
\label{claim-boule}
\int _{\vert z-y \vert \leq R} J(z)\, dz\geq C_N \int _{\vert z\vert \leq R} J(z)\, dz,
\end{equation}
where $0<C_N<1$ is a constant that depends only on the dimension $N$. We postpone the proof of \eqref{claim-boule} and obtain
\begin{equation}
\label{esti2}
\int_{\vert x\vert \leq R} J*u(t,\cdot)\geq m(t)C_N \int _{\vert z\vert \leq R} J(z)\, dz.
\end{equation}
Then, plugging \eqref{esti1} and \eqref{esti2} in \eqref{edo}, we arrive at the differential inequality
$$
\frac{d}{dt}m(t)\geq m(t)\left[\frac{m^{p}(t)}{B_N^{p}R^{Np}}-\left(1-C_N \int_{\vert z\vert \leq R} J(z)\, dz\right)
\right].
$$
Since $m(0)=\int _{\vert x\vert\leq R}u_0=\lambda B_N R^{N}>\left(1-C_N\int_{\vert z\vert \leq R} J(z)\, dz\right)^{1/p}B_N R^{N}$ thanks to \eqref{lambda-R}, the above differential inequality 
enforces\footnote{Indeed, the Bernouilli equation $ \dot{x}=ax^{1+p}-bx$, $a>0$, $b>0$, can be solved explicitly and blows up in finite time as soon as $x(0)>\left(\frac b a\right)^{1/p}$.} the blow up of $m(t)$ in finite time.  This concludes the proof of Theorem \ref{th:vs} $(ii)$.
\end{proof}

For the convenience of the reader, and also to give the exact value of the constant $C_N$, we prove below the rather intuitive claim \eqref{claim-boule}.

\begin{proof}
[Proof of claim \eqref{claim-boule}] In dimension $N=1$, \eqref{claim-boule} clearly holds true with $C_1=\frac 12$ since $J$ is even. Let us now assume $N\geq 2$. We denote by $S_{N-1}$ the unit hypersphere of $\R ^N$. Since $J$ is radial we have
\begin{equation}
\label{radiality}
\int _{\vert z\vert \leq R}J(z)\,dz=\vert S_{N-1}\vert \int _0^R r^{N-1}J(r)\,dr,
\end{equation}
where we recall that $\vert S_{N-1}\vert =\frac{2\pi^{N/2}}{\Gamma(N/2)}$. Let us take $y$ such that $0<r_0:=\vert y\vert <R$. Define $(e_1:=\frac{y}{r_0}, e_2,\cdots,e_N)$ an orthonormal basis of $\R ^N$. For a generic point $z\in \R ^N$ we denote by $(z_1,\cdots,z_N)$ its cartesian coordinates in $(e_1,\cdots,e_N)$ and $(r,\theta _1,\cdots,\theta _{N-1})\in [0,\infty)\times [-\frac \pi 2, \frac \pi 2[^{N-2}\times [-\pi,\pi)$ its polar coordinates, which are related through
\begin{eqnarray*}
z_1&=&r\cos \theta _1\cos \theta _2...\cos \theta _{N-2}\cos \theta _{N-1}\\
z_2&=&r\cos \theta _1\cos \theta _2...\cos \theta _{N-2}\sin \theta _{N-1}\\
z_3&=&r\cos \theta _1\cos \theta _2...\sin \theta _{N-2}\\
...&&\\
z_{N-1}&=&r\cos\theta _1\sin \theta _2\\
z_N&=&r\sin\theta _1.
\end{eqnarray*}
We claim that
\begin{equation}
\label{camenbert}
D:=\left\{z: 0<r<R, \vert \theta _i\vert < \theta ^*:= \arccos \frac{1}{2^{\frac1{N-1}}} \right\}\subset \left\{z: \vert z-y\vert< R\right\}.
\end{equation}
Indeed, for $z\in D$, we have
\begin{eqnarray*}
\vert z-y\vert ^{2}&=&(z_1-r_0)^2+z_2^2+\cdots+z_N^2=r^2-2rr_0\cos \theta _1...\cos \theta _{N-1}+r_0^2\\
&\leq &r^2-2rr_0\cos ^{N-1}\theta ^*+r_0^2=r^2-rr_0+r_0^2\leq \max (r^2,r_0^2)< R^2.
\end{eqnarray*}
It therefore follows from \eqref{camenbert} that
\begin{eqnarray*}
\int _{\vert z-y\vert \leq R}J(z)dz&\geq& \int_D J(z)dz\\
&=& \int _{-\theta ^*}^{\theta ^*} (\cos \theta _1)^{N-2}d\theta _1  \int _{-\theta ^*}^{\theta ^*} (\cos \theta _2)^{N-3}d\theta _2... \int _{-\theta ^{*}}^{\theta ^{*} }d\theta _{N-1}\int _0 ^R r^{N-1}J(r)dr\\
&= & C_N \int _{\vert z\vert \leq R} J(z)dz,
\end{eqnarray*}
in view of \eqref{radiality} and where
\begin{eqnarray*}
C_N:&=&\frac{\int _{-\theta ^*}^{\theta ^*} (\cos \theta _1)^{N-2}d\theta _1  \int _{-\theta ^*}^{\theta ^*} (\cos \theta _2)^{N-3}d\theta _2... \int _{-\theta ^{*}}^{\theta ^{*} }d\theta _{N-1}}{\vert  S_{N-1}\vert}\\
&=&\frac{\int _{-\theta ^*}^{\theta ^*} (\cos \theta _1)^{N-2}d\theta _1  \int _{-\theta ^*}^{\theta ^*} (\cos \theta _2)^{N-3}d\theta _2... \int _{-\theta ^{*}}^{\theta ^{*} }d\theta _{N-1}}{\int _{-\pi/2}^{\pi/2} (\cos \theta _1)^{N-2}d\theta _1  \int _{-\pi/2}^{\pi/2} (\cos \theta _2)^{N-3}d\theta _2... \int _{-\pi}^{\pi }d\theta _{N-1}}\in(0,1),
\end{eqnarray*}
which concludes the proof of \eqref{claim-boule}.
\end{proof}

\section{Hair trigger effect}\label{s:hair-trigger}

\subsection{Hair trigger effect along a subsequence}\label{ss:hair-subsequence}

Following the strategy of \cite[Theorem 18.7]{Qui-Sou-book}, we prove here the hair trigger effect along a subsequence of time.

\begin{proof}[Proof of Corollary \ref{cor:hair} $(i)$] First, let $v_0\in C(\R ^N, [0,1])$ be such that $v_0(x_0+\cdot)$ is radial nonincreasing, for some $x_0\in \R^N$. Let $v(t,x)$ be the global solution of \eqref{eq2} with $v_0$ as initial data. Then, $J$ being radial, $v(t,x_0+\cdot)$ remains radial nonincreasing for later times  $t>0$. Let us prove that 
\begin{equation}
\label{step1}
\limsup _{t\to\infty} v(t,x_0)=1.
\end{equation}
Assume by contradiction that there are $0<\ep<1$ and $T>0$ such that $v(t,x_0)\leq 1-\ep$ for all $t\geq T$, which in turn implies $v(t,x)\leq 1-\ep$, for all $(t,x)\in[T,\infty)\times \R ^N$. As a result
\begin{equation*}
\partial _t v\geq J*v-v+\ep v^{1+p} \quad \text{ in } (T,\infty)\times \R ^N.
\end{equation*}
Hence $w:=\ep ^{1/p}v$ satisfies $\partial _t w\geq J*w-w+w^{1+p}$ in $(T,\infty)\times\R^N$. Since $0<p\leq p_F=\frac \beta N$, it follows from Theorem \ref{th:systematic} and the comparison principle that $w$ is non global, which is a contradiction.

Now, let $u_0:\R^N\to [0,1]$ be as in Corollary \ref{cor:hair} $(i)$, that is continuous and nontrivial. We need to prove
\begin{equation}
\label{hair-limsup}
\limsup _{t\to \infty} \inf _{\vert x\vert \leq R}  u(t,x)=1, \quad \text{ for any } R\geq 0.
\end{equation}
By a time shift if necessary, we can assume further that $u_0>0$. Therefore $u_0$ dominates some $\tilde u _0:\R ^N\to [0,1]$ which is nontrivial and radial nonincreasing. Hence, by comparison, it suffices to prove \eqref{hair-limsup} for $\tilde u(t,x)$ the solution of \eqref{eq2} with $\tilde u_0$ as initial data, for which we can take advantage of the fact that $\tilde u(t,\cdot)$ is radial nonincreasing for later times $t>0$. Again, by a time shift if necessary, we can assume further that $\tilde u _0>0$. Now, for a given $x_0\in \R^N$, $\tilde u_0$ dominates some $v_0\in C(\R ^N,[0,1])$ such that $v_0(x_0+\cdot)$ is radial nonincreasing. It follows from \eqref{step1} and the comparison principle that
$$
\limsup _{t\to \infty} \tilde u (t,x_0)=1.
$$
Since $x_0$ is arbitrary and since $\tilde u(t,\cdot)$ is radial nonincreasing, this implies
$$
\limsup _{t\to \infty} \inf _{\vert x\vert \leq R} \tilde u(t,x)=1, \quad \text{ for any } R\geq 0.
$$
This concludes the proof of \eqref{hair-limsup}. 
\end{proof}

\subsection{Actual hair trigger effect}\label{ss:hair}

In this subsection, we prove the actual hair trigger effect  as stated in Theorem \ref{th:hair}. This requires the combination of an elaborate subsolution involving two different time scales --- see \cite{Zla-05} for a related argument in a local case--- and the following asymptotics for the solution to the linear nonlocal diffusion equation. 

\begin{lem}[The linear equation from below]
\label{lem:pardessous}
Let Assumption \ref{ass:J} hold. For a given $R>0$, let $\varphi(t,x)$ be the solution of $\partial _t \varphi =J*\varphi -\varphi$ with initial data $\varphi _0\equiv \mathbf 1 _{B_R}$. Then there are $\gamma >0$ and $m>0$ such that
$$
\varphi(t,x)\geq  \frac {\gamma}{t^{N/\beta}} \mathbf{1}_{B_{mt^{1/\beta}}}(x), 
$$
for $t>0$ large enough and $x\in\R ^N$.
\end{lem}

\medskip

\begin{proof} This is a direct consequence of \cite{Cha-Cha-Ros-06}. Indeed in virtue of  \cite[Corollary 2.1]{Cha-Cha-Ros-06}, we have
\begin{eqnarray*}
t^{N/\beta}\varphi(t,x)&=&t^{N/\beta}\varphi\left(t,t^{1/\beta}\frac{x}{t^{1/\beta}}\right)-\Vert \varphi _0\Vert _{L^{1}} G_A\left(\frac{x}{t^{1/\beta}}\right)+
\Vert \varphi _0\Vert _{L^{1}} G_A\left(\frac{x}{t^{1/\beta}}\right)\\
&=& o(t)+
\Vert \varphi _0\Vert _{L^{1}} G_A\left(\frac{x}{t^{1/\beta}}\right),
\end{eqnarray*}
as $t\to \infty$, where
$$
G_A(y):=\frac{1}{(2\pi)^{N}}\int _{\R^{N}} e^{iy\cdot\xi}e^{-A\vert \xi \vert ^{\beta}}d\xi.
$$
Noticing that $G_A(y)\to \frac{1}{(2\pi)^{N}}\int _{\R^{N}} e^{-A\vert \xi \vert ^{\beta}}d\xi>0$ as $y\to 0$ enables to select $m>0$ small enough so that, for any $x\in B_{mt^{1/\beta}}$, we have $\Vert \varphi _0\Vert _{L^{1}} G_A\left(\frac{x}{t^{1/\beta}}\right)\geq 2\gamma $ for some $\gamma >0$. This concludes the proof of the lemma.
\end{proof}

\begin{proof}[Proof of Theorem \ref{th:hair}] Let $u_0:\R^N\to [0,1]$ be as in Theorem \ref{th:hair}, that is continuous and nontrivial. Let $\ep >0$ and $R>0$ be given. In view of the hair trigger effect along a subsequence \eqref{hair-limsup} and thanks to the comparison principle, it is enough to consider the solution $u(t,x)$ to \eqref{eq2} with the initial datum  $u_0\equiv (1-\ep)\mathbf 1 _{B_R}$.

Let $\varphi(t,x)$ denote the solution to $\partial _t \varphi=J*\varphi-\varphi$ with initial datum $u_0\equiv (1-\ep)\mathbf 1 _{B_R}$. Notice that, from the comparison principle, we get $\varphi\leq u$. 
From Lemma \ref{lem:pardessous} we deduce that there is $\tau _0>0$ such that 
$$
\varphi(\tau,x)\geq  (1-\ep) \frac{\gamma}{\tau^{N/\beta}}\mathbf 1 _{B_{m\tau ^{1/\beta}}}(x)=:\Phi _0(x), \quad\forall (\tau,x)\in (\tau _0,\infty)\times \R ^N.
$$
Let $\Phi(t,x)$ denote the solution to $\partial _t \Phi=J*\Phi -\Phi$ with initial datum $\Phi_0$. Let $U(t,x)$ denote the solution to $\partial _t U=J*U-U+U^{1+p}(1-U)$ with initial datum $\Phi _0$. Since $U(0,x)=\Phi_0(x)\leq \varphi (\tau,x)\leq u(\tau,x)$, the comparison principle yields
\begin{equation}
\label{U}
u(\tau +t,x)\geq U(t,x),\quad \forall (t,x)\in(0,\infty)\times \R ^N.
\end{equation}

Next, for $X>0$, let us define
$$
w(t,X):=\frac{1}{\left(X^{-p}-\ep pt\right)^{1/p}}=\frac{X}{(1-\ep pt X^{p})^{1/p}},\quad 0<t<\frac{1}{\ep p X^p},
$$
that is the solution to the Cauchy problem 
$$
\partial _t w(t,X) =\ep w^{1+p}(t,X), \quad w(0,X)=X.
$$
Notice that $\partial _Xw=\frac{w^{1+p}}{X^{1+p}}$ and $\partial _{XX} w=(1+p)\frac{w^{1+p}}{X^{2+2p}}(w^p-X^p)\geq 0$ so that $X\mapsto w(t,X)$ is convex.

\begin{lem}[A subsolution] Function
$$
W(t,x):=w(t,\Phi(t,x))
$$
is a subsolution to \eqref{eq2} on $(0,T)\times \R^{N}$, where time 
\begin{equation}
\label{def:T}
T=T(\tau):=\frac{1}{\ep p}\left(\frac{\tau^{pN/\beta}}{(1-\ep)^{p}\gamma ^p}-\frac{1}{(1-\ep)^p}
\right)
\end{equation}
is positive for any $\tau\geq \tau _0$, up to enlarging $\tau _0>0$ if necessary.
\end{lem}

\begin{proof} Since $\Phi(t,x)\leq \Vert \Phi _0\Vert _{L^\infty}=\frac{(1-\ep)\gamma}{\tau ^{N/\beta}}$, we see that $W(t,x)\leq 1-\ep$ on $(0,T)\times \R ^N$. As a result
\begin{eqnarray*}
\partial _t W -(J*W-W)-W^{1+p}(1-W)&\leq &\partial _t W -(J*W-W)-\ep W^{1+p}\\&=& \frac{W^{1+p}}{\Phi ^{1+p}}(J*\Phi-\Phi)-(J*W-W),
\end{eqnarray*}
by a direct computation. On the other hand
\begin{eqnarray*}
(J*W-W)(t,x)&=&\int _{\R ^N}J(y)\left(w(t,\Phi(t,x-y))-w(t,\Phi(t,x))\right)dy\\
&\geq & \int _{\R ^N}J(y) (\Phi(t,x-y)-\Phi(t,x))\frac{w^{1+p}(t,\Phi(t,x))}{\Phi^{1+p}(t,x)}dy\\
&=&\frac{W^{1+p}(t,x)}{\Phi^{1+p}(t,x)}(J*\Phi-\Phi)(t,x),
\end{eqnarray*}
where we have used the convexity of $X\mapsto w(t,X)$. This concludes the proof of the lemma.
\end{proof}

Since $W(0,x)=\Phi_0(x)=U(0,x)$, the comparison principle yields
\begin{eqnarray*}
U(T,x)\geq W(T,x)&=&\left(\Phi(T,x)^{-p}-\ep pT\right)^{-1/p}\\
&=&\left(\Phi(T,x)^{-p}
-\frac{\tau^{pN/\beta}}{(1-\ep)^{p}\gamma ^p}+\frac{1}{(1-\ep)^p}\right)^{-1/p},
\end{eqnarray*}
by definition of time $T$. But, in view of Section \ref{s:basic}, one can write
\begin{eqnarray*}
\Phi (T,x)&=&K(T,\cdot)*\Phi _0 (x)=\frac{(1-\ep)\gamma}{\tau^{N/\beta}}K(T,\cdot)*\mathbf 1 _{B_{m\tau ^{1/\beta}}}(x)\\
&=&\frac{(1-\ep)\gamma}{\tau^{N/\beta}}\left( e^{-T} \mathbf 1 _{B_{m\tau ^{1/\beta}}}(x)+\psi(T,\cdot)*\mathbf 1 _{B_{m\tau ^{1/\beta}}}(x)\right)\\
&=&\frac{(1-\ep)\gamma}{\tau^{N/\beta}}\left(1-e ^{-T}(1-\mathbf 1 _{B_{m\tau ^{1/\beta}}}(x))-\int _{\vert y\vert\geq m\tau ^{1/\beta}}\psi(T,x-y)dy \right),
\end{eqnarray*}
using \eqref{mass-presque-1}. From  now, we restrict ourselves to $x\in B_R$. Hence, up to enlarging $\tau _0>0$ if necessary, we can get rid of the term  $1-\mathbf 1 _{B_{m\tau ^{1/\beta}}}(x)$ for any $\tau \geq \tau _0$ and get
$$
\Phi(T,x)^{-p}=\frac{\tau^{pN/\beta}}{(1-\ep)^{p}\gamma^{p}}\left(1-\int _{\vert y\vert\geq m\tau ^{1/\beta}}\psi(T,x-y)dy \right)^{-p},
$$
so that
\begin{equation}
\label{qqch2}
U(T,x)\geq \left(\frac{1}{(1-\ep)^{p}\gamma ^{p}}\tau ^{pN/\beta}\left[\monstre\right] +\frac{1}{(1-\ep)^{p}}
 \right)^{-1/p}.
 \end{equation}
 
 We now need to estimate $\int _{\vert y\vert\geq m\tau ^{1/\beta}}\psi(T,x-y)dy$. Again, up to enlarging $\tau _0>0$, we have for any $\tau \geq \tau _0$
 \begin{equation}\label{qqch}
 \int _{\vert y\vert \geq m\tau ^{1/\beta}} \psi(T,x-y)dy\leq \int _{\vert z\vert \geq \frac m 2 \tau ^{1/\beta}}\psi (T,z)dz=e^{-T}\sum _{k=1} ^\infty \frac{T^{k}}{k!} \int _{\vert z\vert \geq \frac m 2 \tau ^{1/\beta}}J^{*(k)}(z)dz.
 \end{equation}
 We use the fact that the decay of the kernel $J$ is associated to the behavior of $\widehat J$ near zero. Precisely, quoting \cite[Chapter 2, subsection 2.3.c, (3.5)]{Dur-96}, we get a constant $C>0$ such that, for all $k\geq 1$,
 \begin{eqnarray*}
 \int _{\vert z\vert \geq \frac m 2 \tau ^{1/\beta}}J^{*(k)}(z)dz &\leq & C\tau ^{N/\beta}\int_{\vert \xi\vert \leq \frac 1{C\tau ^{1/\beta}}} (1-\widehat{J^{*(k)}}(\xi))d\xi\\
 &=& C\tau ^{N/\beta}\int_{\vert \xi\vert \leq \frac 1{C\tau ^{1/\beta}}} (1-{\widehat J}\, ^{k}(\xi))d\xi\\
 &\leq & C\tau ^{N/\beta}\int_{\vert \xi\vert \leq \frac 1{C\tau ^{1/\beta}}}k (1-\widehat J(\xi))d\xi.
 \end{eqnarray*}
 Since $1-\widehat J(\xi) \sim A\vert \xi \vert ^\beta$ as $\xi \to 0$, it follows  that, up to enlarging $\tau _0>0$, we have  for all $\tau \geq \tau _0$ and all $k\geq 1$
$$
 \int _{\vert z\vert \geq \frac m 2 \tau ^{1/\beta}}J^{*(k)}(z)dz 
 \leq   k C2A \tau ^{N/\beta}\int _{\vert \xi\vert \leq \frac 1{C\tau ^{1/\beta}}}\vert \xi \vert ^{\beta}d\xi \leq k\frac{C'}{\tau},
 $$
 for some $C'>0$. Plugging this into \eqref{qqch} we get 
 \begin{equation}
 \label{etoile-m}
 \int _{\vert y\vert \geq m\tau ^{1/\beta}} \psi(T,x-y)dy\leq e^{-T}\sum _{k=1}^{+\infty} \frac{T^{k}}{k!}k\frac{C'}{\tau}=C'\frac T \tau.
 \end{equation}
 
 To conclude, the key point is that, in view of \eqref{etoile-m}, \eqref{def:T} and assumption $0<p<\frac 12 \frac \beta N$,
 $$
\tau ^{pN/\beta}\left[\monstre \right]\to 0, \quad \text{ as } \tau \to \infty,
 $$
 uniformly with respect to $x\in B_R$. Hence, in view of \eqref{qqch2} and up to enlarging $\tau _0>0$, we have $U(T,x)\geq 1-2\ep$ for any $\tau \geq \tau _0$, any $x\in B_R$. Hence, we deduce from \eqref{U} that
 $$
 u(\tau+T(\tau),x)\geq 1-2\ep, \quad \forall \tau\geq \tau _0, \forall x \in B_R,
 $$
 which in turn implies
 $$
 u(t,x)\geq 1-2\ep, \quad \forall t\geq t_0:=\tau _0+T(\tau _0), \forall x \in B_R.
 $$
This concludes the proof of Theorem \ref{th:hair}.
\end{proof}

\bibliographystyle{siam}  


\bibliography{biblio}

\end{document}